\documentclass[12pt,a4paper]{article}
\usepackage[cp1250]{inputenc}
\usepackage[T1]{fontenc}
\usepackage[dvips]{graphicx}
\usepackage{epsf}
\usepackage{amssymb}
\usepackage{pstricks}
\usepackage{pst-node}
\usepackage{lineno}
\usepackage[normalem]{ulem}
\usepackage{mathrsfs}
\usepackage{amsmath}
\usepackage{amsthm}
\usepackage{amsfonts}
\usepackage{eepic}
\usepackage{color}
\usepackage{amsthm}

\newtheorem{thm}{Theorem}[section]
\newtheorem{lem}[thm]{Lemma}
\newtheorem{observation}[thm]{Observation}
\newtheorem{cor}[thm]{Corollary}
\newtheorem{prob}[thm]{Problem}
\newtheorem{remark}{Remark}
\newtheorem{claim}{Claim}

\renewcommand\caption[1]{\small\refstepcounter{figure}%
\begin{center}\textbf{Fig.\ \thefigure .}\ #1\end{center}\normalsize}

\newrgbcolor{gray5}{0.5 0.5 0.5}
\newrgbcolor{gray7}{0.7 0.7 0.7}
\newrgbcolor{gray9}{0.9 0.9 0.9}
\newrgbcolor{gray95}{0.95 0.95 0.95}
\newrgbcolor{gray85}{0.85 0.85 0.85}
\newrgbcolor{gray75}{0.75 0.75 0.75}

\def\super{\textrm{s}}
\def\cuper{\textrm{c}}

\def\UU{{\cal U}}\def\FF{{\cal F}}\def\CC{{\cal C}}
\def\EE{{\cal E}}

\usepackage{lineno}

\title{{\textsc{Disjoint dominating and 2-dominating sets in graphs}}}
\date{\today}
\begin{document}

\begin{center}
\noindent\textbf{\Large Disjoint dominating and 2-dominating sets in graphs}
\end{center}

\medskip
\medskip
\medskip
\medskip

\begin{center}{Mateusz Miotk, Jerzy Topp, and Pawe\l{}~\.Zyli\'nski\\[1mm]
University of Gda\'{n}sk, 80-952 Gda\'nsk, Poland\\
{\small \texttt{\{mmiotk,j.topp,zylinski\}@inf.ug.edu.pl}}}\end{center}

\medskip
\medskip
\medskip

\begin{abstract}
\noindent A graph $G$ is a $D\!D_2$-graph if it has a pair $(D,D_2)$ of disjoint sets of vertices of $G$ such that $D$ is a dominating set and $D_2$ is a 2-dominating set of $G$. We provide several characterizations and hardness results concerning $D\!D_2$-graphs.

\medskip
\noindent{\bf Keywords:}  Domination, 2-domination, certified domination, NP-hardness\\
{\bf \AmS \; Subject Classification:} 05C69, 05C85
\end{abstract}

\section{Introduction}

Let $G=(V_G,E_G)$ be a graph. A set of vertices $D \subseteq V_G$ of $G$ is {\it dominating} if every vertex in $V_G-D$  has a~neighbor in $D$, while $D$ is {\it $2$-dominating} if every vertex in $V_G-D$  has at least two neighbors in $D$.
A set $D \subseteq V_G$ is a {\it total dominating set} if every vertex has a~neighbor in $D$. A set $D \subseteq V_G$ is a {\it paired-dominating set} if $D$ is a~total dominating set and the subgraph induced by $D$ contains a perfect matching.
Ore \cite{Ore} was the first to observe that a graph with no isolated vertex contains two disjoint dominating sets. Consequently, the vertex set of a graph without isolated vertices can be partitioned into two dominating sets. Various graph theoretic properties and parameters of graphs having disjoint dominating sets are studied in \cite{AK,HHLMS,HT,HLR09,HM18,KS,LR10}. Characterizations of graphs with disjoint dominating and total dominating sets are given in \cite{HLR10+,HLR10,HLR10++,HS08,HS09,KJ,SH11a}, while in \cite{BDGHHU,DDH,DHH17,DGHH,HY} graphs which have the property that their vertex set can be partitioned into two disjoint total dominating sets are studied.  Conditions which guarantee the existence of a~dominating set whose complement contains a 2-dominating set, a paired-dominating set or an independent dominating set are presented in~\cite{HaynesHenning2005,HLR10,HR13,KJ,KS,SH11b}. In this paper we first restrict our attention to conditions which ensure a partition of vertex set of a graph into a dominating set and a 2-dominating set. The study of graphs having a dominating set whose complement is a 2-dominating set has been initiated by Henning and Rall~\cite{HR13}. They define a {\em $D\!D_2$-pair} in a~graph $G$ to be a~pair $(X,Y)$ of disjoint sets of vertices of $G$ such that $X$ is a dominating set, and $Y$ is a~2-dominating set of $G$. A graph that has a~$D\!D_2$-pair is called a~{\em $D\!D_2$-graph}.  It is easy to observe that a complete graph $K_n$ is a~$D\!D_2$-graph if $n\ge 3$, a~path $P_n$ is a~$D\!D_2$-graph if and only if $n=3$ or $n\ge 5$, while a~cycle $C_n$ is a~$D\!D_2$-graph for every $n\ge 3$. On the other hand, if $G$ is a graph obtained by adding a pendant edge to each vertex of an arbitrary graph $F$, then $G$ is not a~$D\!D_2$-graph, but if $H$ is a~graph obtained by adding at least two pendant edges to each vertex of an arbitrary graph $F$, then $H$ is a $D\!D_2$-graph. Henning and Rall~\cite{HR13} observed that every graph with minimum degree at least two is a $D\!D_2$-graph. They also provided a~constructive characterization of trees that are $D\!D_2$-graphs, and established that the complete bipartite graph $K_{3,3}$ is the only connected graph with minimum degree at least three for which $D \cup D_2$ necessarily contains all vertices of the graph. Herein, we continue their study and complete their structural characterization of all $D\!D_2$-graphs. Next, we focus on minimal $D\!D_2$-graphs and provide the relevant characterization of that class of graphs either. All these results have also algorithmic consequences, leading to simple linear time algorithms for recognizing the two aforementioned graph classes. Next, we study optimization problems related to $D\!D_2$-graphs and non-$D\!D_2$-graphs, respectively. Namely, for a~given $D\!D_2$-graph $G$, the purpose is to find a minimal spanning $D\!D_2$-graph of $G$ of minimum or maximum size. We show that both these problems are NP-hard. Finally, if $G$ is a graph which is not a $D\!D_2$-graph, we consider the question of how many edges must be added to $G$ or subdivided in $G$ to ensure the existence of a $D\!D_2$-pair in the resulting graph. The latter problem turned out to be polynomially tractable, while the former one is NP-hard.

For notation and graph theory terminology we in general follow~\cite{ChLZ16}. Specifically, for a vertex $v$ of a~graph $G=(V_G,E_G)$, its {\em neighborhood\/}, denoted by $N_{G}(v)$, is the set of all vertices adjacent to $v$, and the cardinality of $N_G(v)$, denoted by $d_G(v)$, is called the {\em degree} of~$v$. The {\em closed neighborhood\/} of $v$, denoted by $N_{G}[v]$, is the set $N_{G}(v)\cup \{v\}$. In general, for a subset $X\subseteq V_G$ of vertices, the {\em neighborhood\/} of $X$, denoted by $N_{G}(X)$, is defined to be $\bigcup_{v\in X}N_{G}(v)$, and the {\em closed neighborhood\/} of $X$, denoted by $N_{G}[X]$, is the set $N_{G}(X)\cup X$. The minimum degree of a vertex in $G$ is denoted by $\delta(G)$.  A vertex of degree one is called a {\em leaf}, and the only neighbor of a~leaf is called its {\em support vertex} (or simply, its {\em support}). If a support vertex has at least two leaves as neighbors,  we call it a {\em strong\/} support, otherwise it is a {\em weak} support. The set of leaves, the set of weak supports, the set of strong supports, and the set of all supports of $G$ is denoted by $L_G$, $S'_G$, $S''_G$,  and $S_G$, respectively.

\section{Structural characterization of $D\!D_2$-graphs}\label{sec:DD2}

\noindent In this section we present a structural characterization of $D\!D_2$-graphs. We begin with three useful preliminary results.

\begin{observation} \label{observ-1} {\rm \cite{HR13}}. If $(D,D_2)$ is a $D\!D_2$-pair
in a graph $G$, then every leaf of $G$ belongs to $D_2$, while every support of $G$ belongs to $D$, that is, $L_G\subseteq D_2$ and $S_G\subseteq D$.\end{observation}

\begin{observation} \label{observ-2}  A graph $G$ is a $D\!D_2$-graph if and only if $G$ has a~spanning bipartite subgraph $H=(A,B,E_H)$ such that $d_H(a)\ge 2$ for every $a\in A$, while $d_H(b)\ge 1$ for every $b\in B$.\end{observation}

\begin{proof}  Let $(D,D_2)$ be a $D\!D_2$-pair in $G$. Then the bipartite graph $H=(A,B,E_H)$, where $A=D$, $B=V_G-D$, and $E_H= E_G(D,V_G-D)$ is the set of edges that join a~vertex of $D$ and a vertex of $V_G-D$, is the desired spanning subgraph of $G$.

On the other hand if $H=(A,B,E_H)$ is a bipartite spanning subgraph of $G$ such that $d_H(a)\ge 2$ for every $a\in A$, and $d_H(b)\ge 1$ for every $b\in B$, then $(A,B)$ is a~$D\!D_2$-pair in $H$, and, therefore, in $G$.\end{proof}

From Observation~\ref{observ-2} (or directly from the definition of a $D\!D_2$-graph) we immediately have the following corollary.

\begin{cor}\label{cor:supergraph}
Every spanning supergraph of a $D\!D_2$-graph is a $D\!D_2$-graph.
\end{cor}

Before we state and prove our key characterization of $D\!D_2$-graphs we need some more terminology concerning dominating sets. A~dominating set $D$ of a graph $G$ is said to be {\em certified} if every vertex in $D$ has either zero or at least two neighbors in $V_G - D$ (see~\cite{nasz-AKCE,nasz-OPU}). A vertex $v$ of a graph $G$ is said to be {\em shadowed} with respect to a certified dominating set $D$ if $N_G[v]\subseteq D$. In the next theorem we prove that the $D\!D_2$-graphs are precisely the graphs having a certified dominating set with no shadowed vertex. We also prove that a graph $G$ is a $D\!D_2$-graph if and only if the neighborhood of each weak support of $G$ contains a vertex which is neither a leaf nor a~support vertex.

\begin{thm}\label{first-characterization}
Let G be a graph with no isolated vertex. Then the following three  properties are equivalent:
\begin{itemize}
\item[$(1)$] $G$ has a certified dominating set with no shadowed vertex.
\item[$(2)$] $G$ is a $D\!D_2$-graph.
\item[$(3)$] $N_G(s)-(L_G\cup S_G)\not=\emptyset$ for every weak support $s$ of $G$.
\end{itemize}
\end{thm}
\begin{proof}
The implication $(1) \Rightarrow (2)$ was already observed in \cite{nasz-AKCE}. However, for clarity, we repeat the arguments. Let $D$ be a certified dominating set of $G$ and assume that no element of $D$ is shadowed. Then $D$ is a dominating set of $G$ and $|N_G(x)\cap (V_G-D)|\ge 2$ for every $x\in D$. Consequently, the sets $D$ and $V_G-D$ form a $D\!D_2$-pair in $G$, and so $G$ is a $D\!D_2$-graph.

Assume now that $(D,D_2)$ is a $D\!D_2$-pair of $G$, and let $s$ be a weak support of~$G$. Then $L_G\subseteq D_2$ and $S_G\subseteq D$ (by Observation \ref{observ-1}), and $s\in S_G'\subseteq S_G\subseteq D\subseteq V_G-D_2$. Consequently $N_G(s)\cap (D_2-L_G)\not=\emptyset$ (since $|N_G(s)\cap L_G|=1$ and $|N_G(s)\cap D_2|\ge 2$). This implies that the set $N_G(s)-(L_G\cup S_G)$  is nonempty (as its subset $N_G(x)\cap (D_2-L_G)$ is nonempty) and this establishes the implication $(2) \Rightarrow (3)$.

Finally, assume that  $N_G(s)-(L_G\cup S_G) \not=\emptyset$ for every weak support $s$ of $G$. Let $I$ be a maximal independent subset of $V_G-N_G[S_G]$ in $G$. We claim that $I\cup S_G$ is the desired certified dominating set with no shadowed vertex in $G$. The choice of $I$ and the definition of $N_G[S_G]$ imply that $I\cup S_G$ is a dominating set in $G$. Thus it remains to show that $|N_G(x)-(I\cup S_G)|\ge 2$ for $x\in I\cup S_G$. This is obvious if $x\in I$ or if $x$ is a strong support. If $x$ is a weak support, then also $|N_G(x)-(I\cup S_G)|\ge 2$, as in this case $x$ is adjacent to exactly one leaf (and $L_G\subseteq V_G-(I\cup S_G)$) and, by the assumption, it has another neighbor in $N_G[S_G]-(L_G\cup S_G) \subseteq V_G-(I\cup S_G)$. This completes the proof of the implication $(3) \Rightarrow (1)$. \end{proof}

A graph with minimum degree at least two has no (weak) support vertex and therefore the next corollary is obvious from Theorem~\ref{first-characterization}.

\begin{cor} \label{wniosek-Henning-Rall} {\rm \cite{HR13}}.
Every graph with minimum degree at least two is a  $D\!D_2$-graph.
\end{cor}

\begin{remark} \label{remark1} \rm By Theorem \ref{first-characterization}, a necessary and sufficient condition for a graph 
to be a $D\!D_2$-graph is that each weak support has a neighbor which is neither a leaf nor a~support vertex. Thus, since identifying all leaves and support vertices in a graph of order $n$ and size $m$ can be done in $O(n+m)$ time, {\em the problem of recognizing whether a given graph  is a $D\!D_2$-graph can be solved in linear time}.
\end{remark}

A connected graph $G$ is said to be a {\em minimal $D\!D_2$-graph}, if $G$ is a $D\!D_2$-graph and no proper spanning subgraph of $G$ is a $D\!D_2$-graph. We say that a disconnected graph $G$ is a {\em minimal $D\!D_2$-graph} if every connected component of $G$ is a minimal $D\!D_2$-graph. A multigraph $H$ is called a {\it corona graph} if every vertex of $H$ is a leaf or it is adjacent to a leaf of $H$. The {\em subdivision graph $S(H)$} of a multigraph $H$ is the graph obtained from $H$ by inserting a new vertex onto each edge of $H$. If $e$ is an edge of $H$, then by $n_e$ we denote the vertex inserted onto $e$ in $S(H)$. We describe the structure of minimal $D\!D_2$-graphs in the following theorem.

\begin{thm}\label{thm:DD2-minimal}
A connected graph $G$ is a minimal $D\!D_2$-graph if and only if $G$ is a~star $K_{1,n}$ $(n\ge 2)$, a cycle $C_4$, or $G$ is the subdivision graph of a corona graph, that is, $G=S(H)$ for some connected corona multigraph $H$.
\end{thm}

\begin{proof}  It is easy to check that  $K_{1,n}$ ($n\ge 2$) and $C_4$ are minimal $D\!D_2$-graphs. Now let $H$ be a corona multigraph. It is immediate from Theorem \ref{first-characterization} that $S(H)$ is a~$D\!D_2$-graph. Let $F$ be a $D\!D_2$-graph which is a spanning subgraph of $S(H)$. To prove the minimality of $S(H)$ it suffices to show that every edge of $S(H)$ is in $F$. First, if $vu$ is a pendant edge in $H$, then, since no component of $F$ is of order 1 or 2, each of the edges $vn_{vu}$ and $un_{vu}$ is in $F$. Thus assume that $vu$ is an edge in $H$ and no one of the vertices $v$ and $u$ is a leaf in $H$. Then $v$ and $u$ are supports in $H$, and, therefore, there exist vertices $v'$ and $u'$ such that $vv'$ and $uu'$ are pendant edges in $H$. Since $n_{vu}$ is not an isolated vertex in $F$, at least one of the edges $vn_{vu}$ and $un_{vu}$ is in $F$. It remains to prove that the case where exactly one of the edges $vn_{vu}$ and $un_{vu}$ is in $F$ is impossible. Without loss generality assume that only $un_{vu}$ is in $F$. Then $n_{vu}$ is a leaf in $F$, $u$ is a support in $F$, but now $N_F(n_{uu'}) \subseteq L_F\cup S_F$, and it follows from Theorem~\ref{first-characterization} that $F$ is not a $D\!D_2$-graph. This contradiction proves  the minimality of $S(H)$.

Assume now that $G$ is a connected minimal $D\!D_2$-graph. Then, by Observation \ref{observ-2}, $G$ is a bipartite graph, say $G=(A,B,E_G)$, and without loss of generality we may assume that $d_G(a)\ge 2$ for every $a\in A$, and $d_G(b)\ge 1$ for every $b\in B$. We now consider the following two cases.

{\em Case $1$}: $d_G(a)\ge 3$ for some $a\in A$. If $d_G(b)=1$ for every $b\in N_G(a)$, then $G$ is a~star, $G=K_{1,n}$ ($n=d_G(a)$). It remains to observe that the case $d_G(a)\ge 3$ (for some $a\in A$) and $d_G(b)\ge 2$ for some $b\in N_G(a)$ is impossible, as otherwise it immediately follows from Observation \ref{observ-2} that the proper spanning subgraph $G'=G-ab$ of $G$ would be a~$D\!D_2$-graph, contradicting the minimality of $G$.

{\em Case $2$}: $d_G(a) =2$ for every $a\in A$. Then the set $A$ can be divided into the sets $A_{xy}= N_G(x)\cap N_G(y)$, where $x, y \in B$ and the distance $d_G(x,y)=2$. Now it is straightforward to observe that $G$ is the subdivision graph of the multigraph $H=(V_H, E_H)$ in which $V_H=B$ and there is a one-to-one correspondence between the edges joining vertices $x$ and $y$ in $H$ and the elements of the set $A_{xy}$ in $G$.  If $|A|=1$, then $G=K_{1,2}$ and it is the subdivision graph of a corona graph, as $K_{1,2}= S(P_2)= S(K_1\circ K_1)$. If $|A|=2$, then it follows from the connectivity of $G$ that $G=C_4$ or $G=P_5$ and the last graph is the subdivision graph of a corona graph, as $P_5=S(P_3)= S(K_1\circ 2\,K_1)$. Thus assume that $|A|\ge 3$. Now it remains to prove that $H$ is a corona graph. Suppose to the contrary that $V_H-(L_H \cup S_H) \not= \emptyset$. We consider two subcases.

{\em Subcase $2.1$}. Assume first that there exists $x\in V_H-(L_H\cup S_H)$ such that $N_H(y)-\{x\}\not= \emptyset$ for every $y\in N_H(x)$, in other words, $d_G(y)-|A_{xy}| \ge 1$ for every $y\in N_H(x)$. In this case $d_H(x)>1$ (as $x$ is not a leaf) and $G'= G-\bigcup_{y\in N_H(x)}\{yt\colon t\in A_{xy}\}$ is a proper spanning subgraph of $G$. This graph consists of two vertex-disjoint graphs $G_1=G[N_G[x]]$ and $G_2= G-N_H[x]$.  The fact that $G_1$ is a star of order at least~$3$ implies that $G_1$ is a $D\!D_2$-graph, while the fact that $G_2$ is a $D\!D_2$-graph follows from Observation~\ref{observ-2} and from the choice of $x$ (implying that $d_{G_2}(a)=d_G(a)=2$ for $a\in A\cap V_{G_2}$,  and $d_{G_2}(b) = d_G(b)-|A_{xb}|\ge 1$ for $b \in V_{G_2} \cap B$). Consequently,  $G'=G_1\cup G_2$ is a $D\!D_2$-graph, which contradicts the minimality of $G$.

{\em Subcase $2.2$}. Assume now that for every $x\in V_H-(L_H\cup S_H)$  there exists $y\in N_H(x)$ such that $N_H(y)=\{x\}$. Take any $x_0\in V_H-(L_H\cup S_H)$  and $y_0\in N_H(x)$ such that $N_H(y_0)=\{x_0\}$. In this case there is a multiple edge between $x_0$ and $y_0$. If $N_H(x_0)=\{y_0\}$, then $G=K_{2,n}$ ($n=|A|\ge 3$ is the number of edges joining $x_0$ and $y_0$) and it is a non-minimal $D\!D_2$-graph.
If $N_H(x_0)\not=\{y_0\}$, then the same arguments as in Subcase 2.1 prove that the proper spanning subgraph $G'=G-\{x_0t\colon t\in A_{x_0y_0}\}$ of $G$ is a non-minimal $D\!D_2$-graph. This completes the proof.
\end{proof}

Fig.~\ref{minimal-DD2-graphs} shows a disconnected minimal $D\!D_2$-graph that consists of three types of connected minimal $D\!D_2$-graphs. The last one is the subdivision graph of a corona graph, and one can observe that the main properties of such graphs may be rephrased as in the next observation, which we present without any proof.

\begin{figure}[!h]\begin{center}  \special{em:linewidth 0.4pt} \unitlength 0.3ex \linethickness{0.4pt}
\begin{picture}(220,55)\put(0,0){\multiput(0,10)(0,30){2}{\circle*{2}}
\multiput(30,10)(0,30){2}{\circle*{2}}\path(0,10)(30,10)(30,40)(0,40)(0,10)}
\put(65,0){\multiput(3.91,39.74)(-7.82,0){2}{\circle*{2}}
\path(-3.91,39.74)(0,10)(3.91,39.74)\put(0,10){\circle*{2}}
\multiput(11.5,37.7)(-23,0){2}{\circle*{2}}\path(11.5,37.7)(0,10)(-11.5,37.7)}
\put(100,0){\multiput(5,10)(0,30){2}{\circle*{2}}
\multiput(40,10)(-7.5,30){2}{\circle*{2}}\multiput(40,10)(7.5,30){2}{\circle*{2}}
\multiput(60,25)(0,30){2}{\circle*{2}}\multiput(80,10)(0,30){2}{\circle*{2}}
\multiput(110,10)(0,30){2}{\circle*{2}}\multiput(110,10)(-10,27){2}{\circle*{2}}
\multiput(110,10)(10,27){2}{\circle*{2}}\bezier{200}(5,10)(22.5,20)(40,10)
\bezier{200}(5,10)(22.5,0)(40,10)\path(60,55)(60,25)(40,10)(110,10)
\path(60,25)(80,10)(80,40)\path(5,10)(5,40)\path(32.5,40)(40,10)(47.5,40)
\path(110,10)(110,40)\path(100,37)(110,10)(120,37)\put(5,25){\whiten\circle{2}}
\put(36.25,25){\whiten\circle{2}}\put(43.75,25){\whiten\circle{2}}
\put(22.5,5.0){\whiten\circle{2}}\put(22.5,15.0){\whiten\circle{2}}
\multiput(60,10)(0,30){2}{\whiten\circle{2}}
\multiput(50,17.5)(20,0){2}{\whiten\circle{2}}
\multiput(105,23.5)(10,0){2}{\whiten\circle{2}}
\put(80,25.0){\whiten\circle{2}}\put(110,25.0){\whiten\circle{2}}
\put(95,10){\whiten\circle{2}}
}\end{picture}\caption{A disconnected minimal ${\cal D}\!{\cal D}_2$-graph.} \label{minimal-DD2-graphs}\end{center}\end{figure}

\begin{observation}\label{obs:coco} A connected graph $G$ is the subdivision graph of a~corona graph if and only if $G$ is a bipartite graph, say $G=(A,B,E_G)$, such that $d_G(a)=2$ for every $a\in A$, while every  $b\in B$  is a leaf or it is at the distance two from some leaf of $G$.
\end{observation}

\begin{remark} \label{remark2} {\rm It follows from Observation \ref{observ-2}, Corollary \ref{cor:supergraph}, and Theorem \ref{thm:DD2-minimal} that {\em a~graph $G$ is a ${\cal D}\!{\cal D}_2$-graph if and only if $G$ is a spanning supergraph of a minimal ${\cal D}\!{\cal D}_2$-graph}, that is, $G$ is a spanning supergraph of a graph in which every connected component is a cycle $C_4$, a star $K_{1,n}$ (with $n\ge 2$), or the subdivision graph of a~corona graph. In fact, it is a possible to show a little more: {\em If $G$ is a ${\cal D}\!{\cal D}_2$-graph and $C_4$ is not a connected component of $G$, then $G$ is a spanning supergraph of a graph in which every connected component is a star $K_{1,n}$ ($n\ge 2$) or the subdivision graph of a~corona graph.} Therefore, a tree $T$ is a ${\cal D}\!{\cal D}_2$-graph if and only if $T$ has a spanning forest $F$ in which every connected component is a star $K_{1,n}$ ($n\ge 2$) or the subdivision graph of a~corona tree. Such  trees were also constructively characterized in~\cite{HR13}, using vertex labeling. (Taking into account Observation~\ref{obs:coco}, we point out that if for every connected component of $F$, each of its leaves as well as each vertex at the distance two to a~leaf is assigned the label $B$, while any other vertex is assigned the label $A$, then the resulting vertex labeling of $F$ and so of $T$ either is the one discussed in~\cite{HR13}.)
}\end{remark}

\begin{remark} \label{remark3} {\rm It follows from Theorem \ref{thm:DD2-minimal} that recognizing connected minimal ${\cal D}\!{\cal D}_2$-graphs is relatively easy.
Let $G$ be a connected graph of order $n$ and size $m$. First we check whether $G = C_4$ or $G=K_{1,n}$ for some $n \ge 2$. If $G$ is neither $C_4$ nor $K_{1,n}$, then we check the bipartiteness of $G$ and ad hoc determine the degrees of the vertices in~$G$. If $G$ is a non-bipartite graph or $\delta(G) \ge 2$, then $G$ is not a minimal $D\!D_2$-graph (as it follows from  Theorem~\ref{thm:DD2-minimal}). Thus assume that $G$ is a bipartite graph and  $\delta(G) = 1$. Let $(A,B)$ be the bipartition of $G$. If both $A \cap L_G$ and $B \cap L_G$ are nonempty sets, then $G$ is not a $D\!D_2$-graph  (by Observation \ref{observ-2}). Therefore without loss of generality we may assume that  $L_G \subseteq B$. Now, since $G\not=K_{1,n}$, it must be $d_G(a)= 2$ for every $a \in A$ (as it was observed in Case 1 of the proof of Theorem \ref{thm:DD2-minimal}). By Observation~\ref{obs:coco}, it remains to check whether every vertex in $B-L_G$ is at the distance two from some leaf of $G$. Since each of the above steps can be done in $O(n+m)$ time, we conclude that {\em the problem of recognizing whether a given graph $G$ is a minimal $D\!D_2$-graph can be solved in linear time.}
}\end{remark}

\begin{remark} \label{remark4} {\rm If $G$ is a $D\!D_2$-graph, then let $\gamma\gamma_2(G)$ denote the integer $\min \{|D|+|D_2| : (D,D_2) \textrm{ is a $D\!D_2$-pair in $G$} \}$. This parameter has been defined and studied in \cite{HR13}. It is obvious that if $G$ is a~$D\!D_2$-graph, then $3\le \gamma\gamma_2(G)\le |V_G|$. It is also easy to observe that if $n$ and $k$ are integers such that $3\le k\le n$ and $G_{n,k}$ is one of the graphs  $K_1+\left((k-3)K_1\cup K_{n+2-k} \right)$ and  $K_1+\left((k-3)K_1\cup K_{2,n-k} \right)$, then $|V_{G_{n,k}}| =n$ and $\gamma\gamma_2({G_{n,k}})=k$. In addition, if $G$ is a~graph of order $n$, then
$\gamma\gamma_2(G)=3$ if and only if $K_1+K_{2,n-3}$ is a spanning subgraph of $G$. Henning and Rall in their paper~\cite{HR13}  studied graphs $G$ for which  $\gamma\gamma_2(G) =|V_G|$.  In particular, they observed that there are infinitely many graphs $G$ with the minimum degree two for which  $\gamma\gamma_2(G) =|V_G|$, and proved
that $G=K_{3,3}$ is the only graph with $\delta(G) \ge 3$ for which  $\gamma\gamma_2(G) =|V_G|$. A complete characterization of such graphs remains an open problem. Nevertheless, herein, by establishing the following theorem, we make a small step forward.

\begin{thm}
If $G$ is a minimal $D\!D_2$-graph, then $\gamma\gamma_2(G)=|V_G|$.
\end{thm}

\begin{proof}
Without loss generality assume that $G$ is connected. Then, by Theorem~\ref{thm:DD2-minimal}, $G=K_{1,n}$ ($n\ge 2$), $G=C_4$, or $G=S(H)$ for some connected corona graph $H$. The statement is obvious if $G=K_{1,n}$ ($n\ge 2$) or $G=C_4$. Thus assume that $G=S(H)$, where $H$ is a connected corona graph. Let $(D,D_2)$ be a $D\!D_2$-pair in $G$. If $vu$ is a~pendant edge in $H$, then we must have
$\{v,u\}\subseteq D_2$ and $n_{vu}\in D$. This implies that $V_H\subseteq D$. Finally, if $ab$ is an inner edge in $H$, then $n_{ab}\in D$ (as $D$ is a dominating set in $G=S(H)$ and no neighbor of $n_{ab}$ is in $D$). This proves that $D_2=V_H$ and $D=\{n_{xy}\colon xy\in E_H\}$. Therefore $\gamma\gamma_2(G)=|D|+|D_2|=|V_G|$.
\end{proof}
} \end{remark}

\section{Spanning minimal $D\!D_2$-graphs of a graph}\label{sec:DD2span}

It is obvious that a graph may have many non-isomorphic spanning minimal $D\!D_2$-graphs. For example, the complete graph $G=K_{3n}$ ($n \ge 3$) has a spanning minimal $D\!D_2$-graph being $S(H \circ K_1)$ of size $6n-8$, where $H$ is the multigraph of order 2 in which the only two vertices are joined by $3n-6$ edges (so being maximal in the number of edges over all its spanning minimal $D\!D_2$-graphs), and it also has a~spanning minimal $D\!D_2$-graph that consists of $n$ disjoint $3$-vertex paths of size $2n$ (so being minimal in the number of edges over all its spanning minimal $D\!D_2$-graphs). Therefore, for a~given graph $G$, a natural computational problem is to determine a spanning minimal $D\!D_2$-graph of the minimum size.  Observe that any minimal $D\!D_2$-graph has at least three vertices and the $3$-vertex path is a minimal $D\!D_2$-graph. Therefore, a spanning minimal $D\!D_2$-graph of a graph $G$ of order $n$ must be of size at least $2n/3$. Since the relevant {\em perfect $P_3$-matching problem} is NP-complete even for cubic bipartite planar $2$-connected graphs~\cite{KMZ05}, we  immediately conclude with the following theorem.

\begin{thm}
The problem of determining a spanning minimal $D\!D_2$-graph of the minimum size is NP-hard even for cubic bipartite planar $2$-connected graphs.
\end{thm}

As regards the maximization variant, the problem also remains NP-hard. The idea of our proof is standard and it is based upon reduction from the restricted variant of the $3$-dimensional matching problem~\cite{GJ79,K74}, being a NP-hard problem~\cite{DF86}.

\begin{prob}[The $3\textnormal{DM}3$ problem]
{\em Let $G=(V \cup U, E)$ be a subcubic bipartite planar graph  with no leaf, where $V$ is the union of disjoint sets $X, Y$, and $Z$, where $|X|=|Y|=|Z|=q$, and  every vertex $u \in U$ is adjacent to exactly one vertex from each of the sets $X, Y,$ and $Z$. Is there a subset $U' \subseteq U$ of cardinality $q$ dominating all vertices in $V$?}
\end{prob}

Next, let us define (the decision version of) the {\em maximum spanning minimal $D\!D_2$-graph problem} (the Max-$D\!D_2$ problem for short).

\begin{prob}[The Max-$D\!D_2$ problem]
Let $G$ be a bipartite planar $D\!D_2$-graph, and let $k$ be a positive integer. Does $G$ have a spanning minimal $D\!D_2$-graph of size $k$?
\end{prob}

Let $G=(V \cup U, E)$ be a subcubic bipartite planar graph  with no leaf, where $V$ is the union of disjoint sets $X, Y$, and $Z$, where $|X|=|Y|=|Z|=q$, and  every vertex $u \in U$ is adjacent to exactly one vertex from each of the sets $X, Y,$ and $Z$. Let $G^\super=(V_{G^\super},E_{G^\super})$ be the supergraph of $G$ obtained from $G$ by adjoining to each vertex $v \in X \cup Y \cup Z$, separately, the leaf $l_v$, and to each vertex $u \in U$ ---  the three-vertex path $a_ub_uc_u$, as illustrated in Fig.~\ref{fig:replacement-3DM3}. Formally,  $G^\super$ is the graph in which
$V_{G^\super}=V \cup \{l_v : v \in V\} \cup \bigcup_{u \in U}\{a_u,b_u,c_u\}$, and $E_{G^\super}=E \cup E_V \cup E_U$, where $E_V=\{vl_v: v \in V\}$ and $E_U=\bigcup_{u \in U}\{a_ub_u,b_uc_u,ub_u\}$.
Clearly, the supergraph $G^\super$ has $2|V|+4|U|$ vertices and $|V|+2|E|=|V|+6|U|$ edges, it remains planar, bipartite, and $\Delta(G^\super) = 4$. Furthermore, by Theorem~\ref{first-characterization}, $G^\super$ is a $D\!D_2$-graph (as each weak support of $G^\super$ has a neighbor which is neither a leaf nor a support).

\begin{figure}[!h]
\begin{center}
\pspicture(0,-0.5)(3.4,5)

\scalebox{0.8}{
\cnode[linestyle=dotted,fillstyle=solid,fillcolor=gray85](0,1.5){20pt}{X}
\cnode[linewidth=0.5pt,fillstyle=solid,fillcolor=white](0,1.5){3pt}{xi}
\cnode[linewidth=0.5pt](0,0){3pt}{lxi}
\rput(-.25,1.25){{$x$}}
\rput(-.25,-.25){{$l_x$}}
\rput(0.6,0.6){{$X$}}
\cnode[linestyle=dotted,fillstyle=solid,fillcolor=gray85](2,1.5){20pt}{Y}
\cnode[linewidth=0.5pt,fillstyle=solid,fillcolor=white](2,1.5){3pt}{yi}
\cnode[linewidth=0.5pt](2,0){3pt}{lyi}
\rput(1.75,1.22){{$y$}}
\rput(1.75,-.28){{$l_y$}}
\rput(2.6,0.6){{$Y$}}
\cnode[linestyle=dotted,fillstyle=solid,fillcolor=gray85](4,1.5){20pt}{Z}
\cnode[linewidth=0.5pt,fillstyle=solid,fillcolor=white](4,1.5){3pt}{zi}
\cnode[linewidth=0.5pt](4,0){3pt}{lzi}
\rput(3.75,1.25){{$z$}}
\rput(3.75,-.25){{$l_z$}}
\rput(4.6,0.6){{$Z$}}
\cnode[linestyle=dotted,fillstyle=solid,fillcolor=gray85](2,3.5){20pt}{U}
\cnode[linewidth=0.5pt,fillstyle=solid,fillcolor=white](2,3.5){3pt}{ui}
\rput(1.75,3.75){{$u$}}
\rput(2.6,4.4){{$U$}}
\cnode[linewidth=0.5pt](0.5,5){3pt}{aui}
\cnode[linewidth=0.5pt](2,5){3pt}{bui}
\cnode[linewidth=0.5pt](3.5,5){3pt}{cui}
\rput(0.5,5.35){{$a_u$}}
\rput(2,5.38){{$b_u$}}
\rput(3.5,5.35){{$c_u$}}

\ncline[linewidth=0.5pt]{bui}{ui}
\ncline[linewidth=0.5pt]{aui}{bui}
\ncline[linewidth=0.5pt]{bui}{cui}

\ncline[linewidth=1pt]{xi}{ui}
\ncline[linewidth=1pt]{yi}{ui}
\ncline[linewidth=1pt]{zi}{ui}

\ncline[linewidth=0.5pt]{xi}{lxi}
\ncline[linewidth=0.5pt]{yi}{lyi}
\ncline[linewidth=0.5pt]{zi}{lzi}
}

\endpspicture
\caption{}\label{fig:replacement-3DM3}
\end{center}
\end{figure}

\begin{lem}\label{lem:3DM3-DD2}
There exists a solution to the $3$\textnormal{DM}$3$ problem in $G$ if and only if there exists a solution to the \textrm{Max}-$D\!D_2$ problem in the 
supergraph $G^\super$ with the parameter $k=|V|+3|U|+2q$.
\end{lem}

\begin{proof}
Let $U'$ with $|U'|=q$ be a solution to the $3$DM$3$ problem in $G=(V \cup U, E)$. A spanning minimal $D\!D_2$-graph of $G^\super=(V^\super,E^\super)$ consists of the following $|U|+q$ components $F^3_u, F^7_u, T^4_w$, all being trees  (see Fig.~\ref{fig:replacement-3DM3}), where for $u \in U'$ and for $w \in U-U'$:

\begin{itemize}
\item[--] $F^3_u=(V^3_u,E^3_u)$, where $V^3_u=\{a_u,b_u,c_u\}$, $E^3_u=\{a_ub_u,b_uc_u\}$;
\item[--] $F^7_u=(V^7_u,E^7_u)$, $V^7_u=N_{G}[u] \cup \{l_v : v \in N_{G}(u)\}$, $E^7_u=\bigcup_{v \in N_{G}(u)} \{uv,vl_v\}$;
\item[--] $T^4_w=(V^4_w,E^4_w)$, where $V^4_w=\{w,a_w,b_w,c_w\}$, $E^4_u=\{wb_w,a_wb_w,b_wc_w\}$.
\end{itemize}
Clearly, each of the graphs $F^3_u$, $F^7_u$, and  $T^4_w$ is a minimal $D\!D_2$-graph, and all these graphs constitute a spanning subgraph of $G^\super$ of size $8q+3|U - U'|=5q+3|U|=|V|+3|U|+2q$ as required.

On the other hand, let $H=(V_H,E_H)$ be a spanning minimal $D\!D_2$-graph of size $|V|+3|U|+2q$ of $G^\super$. The following claims are consequences of the structure of $G^\super$ and properties of minimal $D\!D_2$-graphs (see Case 1 in the proof of Theorem~\ref{thm:DD2-minimal}).
\begin{claim}\label{clm:1}
 {\em For a vertex $v \in V$, since $v$ is a weak support in $G^\super$, we have $2 \le d_H(v) \le 3$ and $vl_v \in E_H$. Furthermore, if $d_H(v) =3$ then the connected component of $H$ which $v$ belongs to is isomorphic to $K_{1,3}$.}
\end{claim}
\begin{claim}\label{clm:2}
{\em $H$ is acyclic and has exactly $|U|+q$ connected components.}
\end{claim}
\begin{proof}
Suppose that a connected component $C$ of $H$ has a cycle. Then $d_C(v)=3$ for some $v \in V$, immediately implying $C=K_{1,3}$ by Claim~\ref{clm:1}, a contradiction.  Next, since $H$ is a forest of order $2|V|+4|U|$ and size  $|V|+3|U|+2q$, it has $(2|V|+4|U|)-(|V|+3|U|+2q)=|U|+q$ connected components (as $|V|=3q$).
\end{proof}
\begin{claim}\label{clm:3}
{\em All leaves in $\{l_v : v \in V\}$ belong to at most $q$ components of $H$.}
\end{claim}
\begin{proof}
For a vertex $u \in U$, since $b_u$ is a strong support in $G^\textrm{e}$, we have $2 \le d_H(b_u) \le 3$, both edges $a_ub_u$ and $b_uc_u$ belong to $E_H$, and the connected component of $H$ which $b_u$ belongs to is isomorphic to either $P_3$ or $K_{1,3}$. Therefore, $H$ has at least $|U|$ connected components such that no vertex in $V$ belongs to any of them. Consequently, taking into account Claim~\ref{clm:2}, all leaves in $\{l_v : v \in V\}$ belong to at most $q$ components of $H$ (we shall refer to those components as {\em $l$-components}).
\end{proof}

\begin{claim}\label{clm:4}
{\em For any two vertices $x',x'' \in X$, the leaves $l_{x'}$ and $l_{x''}$, and so both $x'$ and $x''$, belong to two distinct $l$-components of $H$. The analogous properties holds for any two vertices $y',y'' \in Y$ and any two vertices $z',z'' \in Z$.}
\end{claim}

\begin{proof}
Suppose that $C$ is an $l$-component of $H$ which both $l_{x'}$ and $l_{x''}$ belong to. Then the diameter of $C$ is at least six, and $d_C(v)=3$ for some $v \in V$, immediately implying $C=K_{1,3}$ by Claim 1, a contradiction.
\end{proof}

Now, taking into account the structure of the graph~$G$, in particular, the fact that $d_G(u) =3$, by combining Claims~\ref{clm:1},~\ref{clm:3} and~\ref{clm:4}, we may conclude that there are exactly~$q$ $l$-components in $H$, say $C_1,\ldots,C_q$,  such that each component $C_i$ has exactly one vertex from each of the sets $X,Y,Z$ and $U$, say $x_i,y_i,z_i$ and $u_i$, $i=1,\ldots, q$. Since $N_G(u_i)=\{x_i,y_i,z_i\}$ for $i=1,\ldots, q$, the set $U'=\{u_1,\ldots,u_q\}$ constitutes a solution to the $3$DM$3$ problem in $G$.
\end{proof}

Clearly, the above reduction takes polynomial time (in the order and the size of a~graph~$G$). Also, a non-deterministic polynomial algorithm for the Max-$D\!D_2$problem in $G$ just guesses an edge-cover $C$ of $G$ and checks whether $C$ is a $D\!D_2$-graph (which can be done in polynomial time by Remark \ref{remark1}). Hence by Lemma~\ref{lem:3DM3-DD2} and the fact that the $3$DM$3$ problem is NP-complete~\cite{DF86}, we conclude with the following theorem.

\begin{thm}\label{thm-MaxDD2-NP}
The \textnormal{Max-$D\!D_2$} problem in bipartite planar graphs of maximum degree at most four is \emph{NP}-complete. Consequently, the problem of determining the spanning minimal $D\!D_2$-graph of maximum size of bipartite planar graphs of maximum degree at most four is {\em NP}-hard.
\end{thm}

\section{$D\!D_2$-supergraphs of non-$D\!D_2$-graphs}\label{sec:nonDD2}

Staying on the algorithmic issue, given a non-$D\!D_2$-graph $G$, one can ask the following natural question: What is the smallest number of edges which added to $G$ result in a~$D\!D_2$-graph? In particular, one can consider the following decision variant of this problem.
\begin{prob}[The Min-to-$D\!D_2$ problem]
Let $G$ be a non-$D\!D_2$-graph and let $k$ be a positive integer. Is it possible to add at most $k$ edges to $G$ such that the resulting graph becomes a $D\!D_2$-graph?
\end{prob}
\noindent In this section, by reduction from the Set Cover problem~\cite{GJ79,K74}, we show that the Min-to-$D\!D_2$ problem is NP-complete (and so its optimization variant is NP-hard).

Let  $\UU=\{u_1,\ldots,u_n\}$ be a set of $n$ items and let $\FF=\{F_1,\dots,F_m\}$ be a family of $m$ sets containing the items in $\UU$, i.e., each $F_i\subseteq\UU$, such that each element in $\UU$ belongs to at least one set from $\FF$; we assume that $(\UU,\FF)$ is represented as the bipartite graph $G=(V_G,E_G)$ with the partition $V_G=\UU \cup \FF$, where $\{u,F\} \in E_G$ if and only if $u \in F$, for $u \in \UU$ and $F \in \FF$; see Fig.~\ref{fig:set-cover}(a). A~$k$-element subset of $\FF$, whose union is equal to the whole set $\UU$, is called a {\em set cover of size $k$.}

\begin{prob}[The Set Cover problem]
Let $(\UU,\FF)$ be a set system, and let $k \le |\FF|$ be a~positive integer. Does $(\UU,\FF)$ possess a set cover of size $k$?
\end{prob}

The Set Cover problem is well known to be NP-complete~\cite{K74}. We are going to prove that for a given set system  $(\UU,\FF)$, represented as the  bipartite graph $G=(\UU \cup \FF,E_G)$, and a positive integer $k \le |\FF|$, there exists a set cover of size $k$ if and only if there is a solution for the Min-to-$D\!D_2$ problem in the graph $G^\cuper=(V_{G^\cuper},E_{G^\cuper})$ (see Fig.~\ref{fig:set-cover} for an illustration), with the same parameter $k$, where:
\begin{itemize}
\item[$\bullet$ ] $V_{G^\cuper}=\FF \cup \bigcup_{i=1}^\eta \{u^i_1,\ldots,u^i_n\} \cup \bigcup_{i=1}^\eta \{l^i_1,\ldots,l^i_n\} \cup \bigcup_{j=1}^m \{L^j_1,L^j_2\}$, with $\eta=2k+1$;
\item[$\bullet$ ] $E_{G^\cuper}=\check{E} \cup \bar{E}$, where $\check{E} = \bigcup_{j=1}^m\bigcup_{u_i \in F_j}\{u_i^1F_j,u_i^2F_j,\ldots,u_i^\eta F_j\}$,\\
and $\bar{E}=\bigcup_{i=1}^\eta \{u^i_1l^i_1,u^i_2l^i_2,\ldots,u^i_nl^i_n\} \cup \bigcup_{j=1}^m \{F_jL^j_1,F_jL^j_2\}$.
\end{itemize}
Notice that $L_{G^\cuper}= \bigcup_{i=1}^\eta \{l^i_1,l^i_2,\ldots,l^i_n\} \cup \bigcup_{j=1}^m \{L^j_1,L^j_2\}$,  and $S'_{G^\cuper}=\bigcup_{i=1}^\eta \{u^i_1,u^i_2,\ldots,u^i_n\}$, while $S''_{G^\cuper}=\FF$. Moreover, none of weak supports in $G^\cuper$ has a non-leaf non-support neighbor, i.e., we have $N_{G^\cuper}(s) - L_{G^\cuper} =\{F \in \FF: s \in F\} \subseteq S''_{G^\cuper}$  for each weak support $s \in S'_{G^\cuper}$, and hence $G^\cuper$ is not a $D\!D_2$-graph by Theorem~\ref{first-characterization}. In addition,  the set $S'_{G^\cuper}$ of weak supports is independent in $G^\cuper$. Clearly, our reduction takes polynomial time: the order of $G^\cuper$ is equal to $(2k+1)n+3m \le (2m+1)n+3m$, while its size equals $2m+(2k+1)(n+\sum_{j=1}^m|F_j|) \le 2m+(2m+1)(n+\sum_{j=1}^m|F_j|)$.

\begin{figure}[tb]
\begin{center}
\pspicture(0,0.35)(8,4.5)
\scalebox{0.8}{
\rput(-0.25,5.75){a)}

\cnode[linewidth=0.5pt,fillstyle=solid,fillcolor=white](2,1){3pt}{F1}
\cnode[linewidth=0.5pt,fillstyle=solid,fillcolor=white](3.25,1){3pt}{F2}
\cnode[linewidth=0.5pt,fillstyle=solid,fillcolor=white](5.75,1){3pt}{Fm}

\cnode[linewidth=0.5pt,fillstyle=solid,fillcolor=white](0.75,4){3pt}{u1}
\cnode[linewidth=0.5pt,fillstyle=solid,fillcolor=white](2,4){3pt}{u2}
\cnode[linewidth=0.5pt,fillstyle=solid,fillcolor=white](3.25,4){3pt}{u3}
\cnode[linewidth=0.5pt,fillstyle=solid,fillcolor=white](5.75,4){3pt}{un1}
\cnode[linewidth=0.5pt,fillstyle=solid,fillcolor=white](7,4){3pt}{un}

\ncline[linewidth=0.5pt,arrowsize=4pt 2,linecolor=black]{F1}{u1}
\ncline[linewidth=0.5pt,arrowsize=4pt 2,linecolor=black]{F1}{u3}
\ncline[linewidth=0.5pt,arrowsize=4pt 2,linecolor=black]{F2}{u2}
\ncline[linewidth=0.5pt,arrowsize=4pt 2,linecolor=black]{F2}{u3}
\ncline[linewidth=0.5pt,arrowsize=4pt 2,linecolor=black]{F2}{un1}
\ncline[linewidth=0.5pt,arrowsize=4pt 2,linecolor=black]{Fm}{u3}
\ncline[linewidth=0.5pt,arrowsize=4pt 2,linecolor=black]{Fm}{un}

\rput(2,0.5){$F_1$}
\rput(3.25,0.5){$F_2$}
\rput(5.75,0.5){$F_m$}

\rput(0.75,4.35){$u_1$}
\rput(2,4.35){$u_2$}
\rput(3.25,4.35){$u_3$}
\rput(5.75,4.35){$u_{n-1}$}
\rput(7,4.35){$u_n$}

\rput(4.5,4){$\cdots$}
\rput(4.5,1){$\cdots$}

\rput(3.8125,-0.25){a set system $(\UU,\FF)$ as the bipartite graph $G$}

}
\endpspicture
\pspicture(0,0.35)(6.6,4.8)
\scalebox{0.8}{
\rput(-0.75,5.75){b)}

\pspolygon[linewidth=0.5pt,arrowsize=4pt 2,linecolor=gray95,fillstyle=solid,fillcolor=gray95](1.5,1)(6.25,1)(7.75,4)(7.75,5.5)(0,5.5)(0,4)

\cnode[linewidth=0.5pt,fillstyle=solid,fillcolor=white](2,1){3pt}{F1}
\cnode[linewidth=0.5pt,fillstyle=solid,fillcolor=white](3.25,1){3pt}{F2}
\cnode[linewidth=0.5pt,fillstyle=solid,fillcolor=white](5.75,1){3pt}{Fm}

\cnode[linewidth=0.5pt,fillstyle=solid,fillcolor=white](0.75,4){3pt}{u1}
\cnode[linewidth=0.5pt,fillstyle=solid,fillcolor=white](2,4){3pt}{u2}
\cnode[linewidth=0.5pt,fillstyle=solid,fillcolor=white](3.25,4){3pt}{u3}
\cnode[linewidth=0.5pt,fillstyle=solid,fillcolor=white](5.75,4){3pt}{un1}
\cnode[linewidth=0.5pt,fillstyle=solid,fillcolor=white](7,4){3pt}{un}

\cnode[linewidth=0.5pt,fillstyle=solid,fillcolor=white](0.75,5.5){3pt}{lu1}
\cnode[linewidth=0.5pt,fillstyle=solid,fillcolor=white](2,5.5){3pt}{lu2}
\cnode[linewidth=0.5pt,fillstyle=solid,fillcolor=white](3.25,5.5){3pt}{lu3}
\cnode[linewidth=0.5pt,fillstyle=solid,fillcolor=white](5.75,5.5){3pt}{lun1}
\cnode[linewidth=0.5pt,fillstyle=solid,fillcolor=white](7,5.5){3pt}{lun}

\ncline[linewidth=0.5pt,arrowsize=4pt 2,linecolor=black]{F1}{u1}
\ncline[linewidth=0.5pt,arrowsize=4pt 2,linecolor=black]{F1}{u3}
\ncline[linewidth=0.5pt,arrowsize=4pt 2,linecolor=black]{F2}{u2}
\ncline[linewidth=0.5pt,arrowsize=4pt 2,linecolor=black]{F2}{u3}
\ncline[linewidth=0.5pt,arrowsize=4pt 2,linecolor=black]{F2}{un1}
\ncline[linewidth=0.5pt,arrowsize=4pt 2,linecolor=black]{Fm}{u3}
\ncline[linewidth=0.5pt,arrowsize=4pt 2,linecolor=black]{Fm}{un}

\ncline[linewidth=0.5pt,arrowsize=4pt 2,linecolor=black]{lu1}{u1}
\ncline[linewidth=0.5pt,arrowsize=4pt 2,linecolor=black]{lu2}{u2}
\ncline[linewidth=0.5pt,arrowsize=4pt 2,linecolor=black]{lu3}{u3}
\ncline[linewidth=0.5pt,arrowsize=4pt 2,linecolor=black]{lun1}{un1}
\ncline[linewidth=0.5pt,arrowsize=4pt 2,linecolor=black]{lun}{un}

\rput(2,0.5){$F_1$}
\rput(3.25,0.5){$F_2$}
\rput(5.75,0.5){$F_m$}

\rput(1.15,4.35){$u^i_1$}
\rput(2.4,4.35){$u^i_2$}
\rput(3.65,4.35){$u^i_3$}
\rput(5.15,4.35){$u^i_{n-1}$}
\rput(6.6,4.35){$u^i_n$}

\rput(1.15,5.15){$l^i_1$}
\rput(2.4,5.15){$l^i_2$}
\rput(3.65,5.15){$l^i_3$}
\rput(5.15,5.15){$l^i_{n-1}$}
\rput(6.6,5.15){$l^i_n$}

\rput(4.5,4){$\cdots$}
\rput(4.5,1){$\cdots$}


\rput(3.8125,-0.25){structure $X_i$}

}

\endpspicture

\pspicture(0,-1.25)(8,7)
\scalebox{0.8}{
\rput(-1.75,6.25){c)}

\pspolygon[linewidth=0.5pt,arrowsize=4pt 2,linecolor=gray75,fillstyle=solid,fillcolor=gray75](1.75,1)(8.25,1)(2.25,4)(2.25,5.5)(-0.75,5.5)(-.75,4)(1.75,1)

\pspolygon[linewidth=0.5pt,arrowsize=4pt 2,linecolor=gray85,fillstyle=solid,fillcolor=gray85](1.75,1)(8.25,1)(5.75,4)(5.75,5.5)(2.75,5.5)(2.75,4)(1.75,1)

\pspolygon[linewidth=0.5pt,arrowsize=4pt 2,linecolor=gray95,fillstyle=solid,fillcolor=gray95](1.75,1)(8.25,1)(10.75,4)(10.75,5.5)(7.75,5.5)(7.75,4)(1.75,1)

\rput(6.75,4){$\cdots$}
\rput(6,1){$\cdots$}

\rput(1,3.985){\small $\cdots$}
\rput(4.5,3.985){\small $\cdots$}
\rput(9.5,3.985){\small $\cdots$}

\rput(.75,5.9){structure $X_1$}
\rput(4.25,5.9){structure $X_2$}
\rput(9.25,5.9){structure $X_\eta$}

\rput(1.5,0.7){$F_1$}
\rput(3.5,0.7){$F_2$}
\rput(8.5,0.7){$F_m$}

\rput(1.5,-1){$L^1_1$}
\rput(2.5,-1){$L^1_2$}
\rput(3.5,-1){$L^2_1$}
\rput(4.5,-1){$L^2_2$}
\rput(7.5,-1){$L^m_1$}
\rput(8.5,-1){$L^m_2$}

\rput(11.5,2){the final graph $G^\cuper$}

\psline[linewidth=0.5pt,arrowsize=4pt 2,linecolor=black]{-}(2,1)(-0.5,4)(-0.5,5.5)
\psline[linewidth=0.5pt,arrowsize=4pt 2,linecolor=black]{-}(2,1)(0.5,4)

\psline[linewidth=0.5pt,arrowsize=4pt 2,linecolor=black]{-}(2,1)(3,4)(3,5.5)
\psline[linewidth=0.5pt,arrowsize=4pt 2,linecolor=black]{-}(2,1)(4,4)

\psline[linewidth=0.5pt,arrowsize=4pt 2,linecolor=black]{-}(2,1)(8,4)(8,5.5)
\psline[linewidth=0.5pt,arrowsize=4pt 2,linecolor=black]{-}(2,1)(9,4)

\psline[linewidth=0.5pt,arrowsize=4pt 2,linecolor=black]{-}(1.5,-0.5)(2,1)(2.5,-0.5)

\cnode[linewidth=0.5pt,fillstyle=solid,fillcolor=white](2,1){3pt}{un}
\cnode[linewidth=0.5pt,fillstyle=solid,fillcolor=white](-0.5,4){3pt}{un}
\cnode[linewidth=0.5pt,fillstyle=solid,fillcolor=white](3,4){3pt}{un}
\cnode[linewidth=0.5pt,fillstyle=solid,fillcolor=white](8,4){3pt}{un}
\cnode[linewidth=0.5pt,fillstyle=solid,fillcolor=white](-0.5,5.5){3pt}{un}
\cnode[linewidth=0.5pt,fillstyle=solid,fillcolor=white](3,5.5){3pt}{un}
\cnode[linewidth=0.5pt,fillstyle=solid,fillcolor=white](8,5.5){3pt}{un}
\cnode[linewidth=0.5pt,fillstyle=solid,fillcolor=white](2.5,-0.5){3pt}{un}
\cnode[linewidth=0.5pt,fillstyle=solid,fillcolor=white](1.5,-.5){3pt}{un}

\psline[linewidth=0.5pt,arrowsize=4pt 2,linecolor=black]{-}(4,1)(0,4)(0,5.5)
\psline[linewidth=0.5pt,arrowsize=4pt 2,linecolor=black]{-}(4,1)(0.5,4)
\psline[linewidth=0.5pt,arrowsize=4pt 2,linecolor=black]{-}(4,1)(1.5,4)(1.5,5.5)

\psline[linewidth=0.5pt,arrowsize=4pt 2,linecolor=black]{-}(4,1)(3.5,4)(3.5,5.5)
\psline[linewidth=0.5pt,arrowsize=4pt 2,linecolor=black]{-}(4,1)(4,4)
\psline[linewidth=0.5pt,arrowsize=4pt 2,linecolor=black]{-}(4,1)(5,4)(5,5.5)

\psline[linewidth=0.5pt,arrowsize=4pt 2,linecolor=black]{-}(4,1)(8.5,4)(8.5,5.5)
\psline[linewidth=0.5pt,arrowsize=4pt 2,linecolor=black]{-}(4,1)(9,4)
\psline[linewidth=0.5pt,arrowsize=4pt 2,linecolor=black]{-}(4,1)(10,4)(10,5.5)

\psline[linewidth=0.5pt,arrowsize=4pt 2,linecolor=black]{-}(3.5,-0.5)(4,1)(4.5,-0.5)

\cnode[linewidth=0.5pt,fillstyle=solid,fillcolor=white](4,1){3pt}{un}
\cnode[linewidth=0.5pt,fillstyle=solid,fillcolor=white](10,4){3pt}{un}
\cnode[linewidth=0.5pt,fillstyle=solid,fillcolor=white](8.5,4){3pt}{un}
\cnode[linewidth=0.5pt,fillstyle=solid,fillcolor=white](5,4){3pt}{un}
\cnode[linewidth=0.5pt,fillstyle=solid,fillcolor=white](3.5,4){3pt}{un}
\cnode[linewidth=0.5pt,fillstyle=solid,fillcolor=white](1.5,4){3pt}{un}
\cnode[linewidth=0.5pt,fillstyle=solid,fillcolor=white](0.5,4){3pt}{un}
\cnode[linewidth=0.5pt,fillstyle=solid,fillcolor=white](0,4){3pt}{un}
\cnode[linewidth=0.5pt,fillstyle=solid,fillcolor=white](10,5.5){3pt}{un}
\cnode[linewidth=0.5pt,fillstyle=solid,fillcolor=white](8.5,5.5){3pt}{un}
\cnode[linewidth=0.5pt,fillstyle=solid,fillcolor=white](5,5.5){3pt}{un}
\cnode[linewidth=0.5pt,fillstyle=solid,fillcolor=white](3.5,5.5){3pt}{un}
\cnode[linewidth=0.5pt,fillstyle=solid,fillcolor=white](1.5,5.5){3pt}{un}
\cnode[linewidth=0.5pt,fillstyle=solid,fillcolor=white](0.5,5.5){3pt}{un}
\cnode[linewidth=0.5pt,fillstyle=solid,fillcolor=white](0,5.5){3pt}{un}
\cnode[linewidth=0.5pt,fillstyle=solid,fillcolor=white](4.5,-0.5){3pt}{un}
\cnode[linewidth=0.5pt,fillstyle=solid,fillcolor=white](3.5,-.5){3pt}{un}

\psline[linewidth=0.5pt,arrowsize=4pt 2,linecolor=black]{-}(8,1)(0.5,4)(0.5,5.5)
\psline[linewidth=0.5pt,arrowsize=4pt 2,linecolor=black]{-}(8,1)(2,4)(2,5.5)

\psline[linewidth=0.5pt,arrowsize=4pt 2,linecolor=black]{-}(8,1)(4,4)(4,5.5)
\psline[linewidth=0.5pt,arrowsize=4pt 2,linecolor=black]{-}(8,1)(5.5,4)(5.5,5.5)

\psline[linewidth=0.5pt,arrowsize=4pt 2,linecolor=black]{-}(8,1)(9,4)(9,5.5)
\psline[linewidth=0.5pt,arrowsize=4pt 2,linecolor=black]{-}(8,1)(10.5,4)(10.5,5.5)

\psline[linewidth=0.5pt,arrowsize=4pt 2,linecolor=black]{-}(8.5,-0.5)(8,1)(7.5,-0.5)

\cnode[linewidth=0.5pt,fillstyle=solid,fillcolor=white](8,1){3pt}{un}
\cnode[linewidth=0.5pt,fillstyle=solid,fillcolor=white](10.5,4){3pt}{un}
\cnode[linewidth=0.5pt,fillstyle=solid,fillcolor=white](9,4){3pt}{un}
\cnode[linewidth=0.5pt,fillstyle=solid,fillcolor=white](5.5,4){3pt}{un}
\cnode[linewidth=0.5pt,fillstyle=solid,fillcolor=white](4,4){3pt}{un}
\cnode[linewidth=0.5pt,fillstyle=solid,fillcolor=white](2,4){3pt}{un}
\cnode[linewidth=0.5pt,fillstyle=solid,fillcolor=white](.5,4){3pt}{un}
\cnode[linewidth=0.5pt,fillstyle=solid,fillcolor=white](10.5,5.5){3pt}{un}
\cnode[linewidth=0.5pt,fillstyle=solid,fillcolor=white](9,5.5){3pt}{un}
\cnode[linewidth=0.5pt,fillstyle=solid,fillcolor=white](5.5,5.5){3pt}{un}
\cnode[linewidth=0.5pt,fillstyle=solid,fillcolor=white](4,5.5){3pt}{un}
\cnode[linewidth=0.5pt,fillstyle=solid,fillcolor=white](2,5.5){3pt}{un}
\cnode[linewidth=0.5pt,fillstyle=solid,fillcolor=white](.5,5.5){3pt}{un}
\cnode[linewidth=0.5pt,fillstyle=solid,fillcolor=white](8.5,-0.5){3pt}{un}
\cnode[linewidth=0.5pt,fillstyle=solid,fillcolor=white](7.5,-.5){3pt}{un}
}

\endpspicture

\caption{}\label{fig:set-cover}
\end{center}
\end{figure}
\begin{lem}\label{lem:SC-Min-to-DD2}
Let $\langle G=(\UU \cup \FF,E),k\rangle$ be an instance of the Set Cover problem $($with $k \le |\FF|$$)$. Then for $(\UU,\FF)$ there exists a set cover of size $k$ if and only if there exists a~solution to the Min-to-$D\!D_2$ problem for the instance $\langle G^\cuper,k\rangle$.
\end{lem}

\begin{proof}
Without loss of generality assume that $\CC=\{F_1,F_2,\ldots,F_k\}$ is a solution to the Set Cover problem for $\langle G,k\rangle$.  Let $G^\cuper_k$ be the graph resulting from $G^\cuper$ by adding $k$ edges $L^1_1L^1_2,L^2_1L^2_2,\ldots,L^k_1L^k_2$. Since $\CC$ is a set cover of $G$, we have  $N_{G^\cuper_k}(s)-(L_{G^\cuper_k}\cup S_{G^\cuper_k})\not=\emptyset$ for every weak support $s \in S'_{G^\cuper_k}$, and hence $G^\cuper_k$ is a $D\!D_2$-graph by Theorem~\ref{first-characterization}.

On the other hand, let $\EE=\{e_1,\ldots,e_k\}$ be a solution to the Min-to-$D\!D_2$ problem in~$G^\cuper$. If $k = |\FF|$, then $\FF$ itself constitutes the required set cover for $(\UU,\FF)$. Thus assume $k \le |\FF|$. Let $U_i=\{u^i_1,\ldots,u^i_n\}$ be the subset of weak supports in $G^\cuper$, $i=1,\ldots,\eta$. In order to be a~$D\!D_2$-graph by $G^\cuper_\EE=G^\cuper+\EE \ (=G^\cuper_\EE +\{e_1,\ldots,e_k\})$, taking into account Theorem~\ref{first-characterization} and the structure of $G^\cuper$, in particular, $\eta=2k+1$ and the fact that all its weak supports are pairwise independent, it follows that for each $U_i$, there exists $j_i$ such that  the vertex $F_{j_i}$ is not a support vertex in $G^\cuper_\EE$ any longer. Let $J=\{j_i : i \in \{1,2,\ldots, \eta\}\}$. Next, observe that to make $F_{j_i}$ a non-support vertex in $G^\cuper_\EE$, there must be an edge in $\EE$ incident to the leaf $L^{j_i}_1$ and an edge in $\EE$ incident to the leaf $L^{j_i}_2$.

If $\EE$ is {\em consistent}, that is, if each of the edges in $\EE$ is of the form $L^j_1L^j_2$ for some $j \in \{1,2,\ldots,m\}$, then the family $\{F_{j} : j \in J\}$ constitutes a solution of size at most~$k$ to the Set Cover problem for $\langle G,k\rangle$, as $|J|=k$.

Therefore assume that $\EE$ is not consistent. The idea is to replace $\EE$ with another set of at most $k$ edges which also is a solution to the Min-to-$D\!D_2$-problem for $\langle G^\cuper,k\rangle$, but which is consistent. Our replacement it is based upon the following simple claim.
\begin{claim}\label{clm:5}
{\em Let $H=(V_H,E_H)$ be a graph with $\delta(H) \ge 1$. If $S$ is a set of disjoint pairs of vertices in $H$, then $|S| \le |E_H|$.}
\end{claim}
In particular, considering $H$ as the graph whose edge set is $\EE$ and vertex set is the set of all endpoints of edges in $\EE$, and setting $S=\bigcup_{j \in J}\{(L^j_1,L^j_2)\}$, we obtain $|J|=|S| \le k$. Therefore, the set $\EE'=\bigcup_{j \in J} \{L^j_1L^j_2\}$ is of size at most $k$ and, by the definition of the set $J$, it also constitutes a solution to the Min-to-$D\!D_2$ problem for $\langle G^\cuper,k\rangle$ (as each edge $L^j_1L^j_2$  makes $F_{j}$ a non-support vertex in $G^\cuper + \EE'$, for each $j \in J$). Since $\EE'$ is consistent from the definition, the family $\{F_{j} : j \in J\}$ is a solution of size at most~$k$ to the Set Cover problem for $\langle G,k\rangle$ as $|J| \le k$.
\end{proof}

A non-deterministic polynomial algorithm for the Min-to-$D\!D_2$ problem in $G$ just guesses at most $k$ ``missing'' edges and checks whether the graph resulting from adding these edges to $G$ is a $D\!D_2$-graph (which can be done in polynomial time by Remark~\ref{remark1}). Hence by Lemma~\ref{lem:SC-Min-to-DD2} and the fact that the Set Cover problem is NP-complete~\cite{K74}, we obtain the following theorem.

\begin{thm}
The Min-to-$D\!D_2$ problem is NP-complete. Consequently, given a non-$D\!D_2$-graph $G$, the problem of determining the minimum number of missing edges after adding which the resulting graph becomes a $D\!D_2$-graph is NP-hard.
\end{thm}

Finally, since being a $D\!D_2$-graph is closely related to the operation of subdivision an edge (see Theorem~\ref{thm:DD2-minimal}), given a non-$D\!D_2$-graph, one can also ask for the minimum number of such operations  after applying which the resulting graph becomes a $D\!D_2$-graph. Fortunately, in this case the problem is polynomially tractable, by a simple reduction to the maximum matching problem. Namely, let $X_G$ denote the set of all weak supports of a graph $G$ such that $N_G(s)-(L_G\cup S_G)=\emptyset$ for every $s \in X_G$ (if $G$ is a $D\!D_2$-graph, then $X_G$ is the empty set).  Consider now the graph $H$ resulting from subdivision of an edge $e \in E_G$. Observe that $X_H \subseteq X_G$ and $|X_H| \ge |X_G|-2$, in other words, subdividing $e$ excludes at most two vertices in $X_G$. In particular, $|X_H| = |X_G|-2$ if and only if $e=s_1s_2$ for some $s_1,s_2 \in X_G$. Consequently, the minimum number of edge subdivisions is equal to $|X_G|-|M|$, where $M$ is the maximum matching in the induced subgraph $G[X_G]$, which immediately results in the following theorem.

\begin{thm}
Given a non-$D\!D_2$-graph $G$, the minimum number of edge subdivision in $G$ after applying which the resulting graph becomes a $D\!D_2$-graph can be computed in polynomial time.
\end{thm}

\section{Closing open problems}
We close this paper with the following list of open problems that we have yet to settle.

\begin{prob}
	Characterize the class of spanning supergraphs of minimal $D\!D_2$-graphs $G$ for which $\gamma\gamma_2(G) = |V_G|$, see Remark~\ref{remark4}.
\end{prob}

\begin{prob}
	Characterize the class of spanning supergraphs of minimal $D\!D_2$-graphs for which the smallest dominating set is a certified dominating set, see the definition before Theorem~\ref{first-characterization}.
\end{prob}

There are also some algorithmic issues is related to the optimization problems discussed in Sections~\ref{sec:DD2span} and~\ref{sec:nonDD2}. For example, observe that by adding edges connecting distinct pairs of leaves that are adjacent to weak supports, we obtain a $D\!D_2$-graph. However, for the corona graph $G$ of a star $K_{1,n}$ of order at least three --- notice that such $G$ is not a $D\!D_2$-graph --- it is enough to add only one properly chosen edge of this type. Therefore, we state the following three problems.

\begin{prob}
Provide an efficient approximation algorithm for the problem of determining:
\begin{itemize}
    \item[$(1)$] a spanning minimal $D\!D_2$-graph of the minimum size;
    \item[$(2)$] a spanning minimal $D\!D_2$-graph of the maximum size;
    \item[$(3)$] the smallest number of edges which added to a graph result in a $D\!D_2$-graph.
\end{itemize}
\end{prob}

\end{document}